\documentclass{article}

\usepackage{amsmath, amsthm}%
\usepackage{amsfonts}%
\usepackage{amssymb}%
\usepackage{graphicx}
\usepackage{amscd, amstext, hyperref}
\usepackage{pictexwd}
\usepackage{dcpic}
\usepackage[mathscr]{eucal}

\theoremstyle{definition}
\newtheorem{theorem}{Theorem}[section]
\newtheorem{proposition}[theorem]{Proposition}
\newtheorem{remark}[theorem]{Remark}
\newtheorem{example}[theorem]{Example}
\newtheorem{definition}[theorem]{Definition}

\newtheorem{lemma}[theorem]{Lemma}

\begin{document}

\title{Reflexive Polytopes and Lattice-Polarized K3 Surfaces}
\author{Ursula Whitcher}
\date{}
\maketitle

\abstract{In this expository note, we review the standard formulation of mirror symmetry for Calabi-Yau hypersurfaces in toric varieties, and compare this construction to a description of mirror symmetry for K3 surfaces which relies on a sublattice of the Picard lattice.  We then show how to combine information about the Picard group of a toric ambient space with data about automorphisms of the toric variety to identify families of K3 surfaces with high Picard rank.}

\section{Mirror symmetry} 

String theory posits that our universe consists of four space-time dimensions together with six extra, compact real dimensions which take the shape of a Calabi-Yau variety.  Physicists have defined multiple consistent theories, which use different information about the underlying varieties.  These theories are linked by \emph{dualities}, which transform physical observables described by one collection of geometric data into equivalent observables described by different geometric data.  Attempts to build a mathematically consistent description of the duality between ``Type IIA'' and ``Type IIB'' string theories led to the thriving field of \emph{mirror symmetry}, which is based on the philosophy that the complex moduli of a given family of Calabi-Yau varieties should correspond to the complexified K\"{a}hler moduli of a ``mirror'' family.  

Physical models typically focus on the properties of Calabi-Yau threefolds.  Calabi-Yau threefolds fit naturally into a ladder of varieties of increasing dimension, all with trivial canonical bundle.  In one (complex) dimension, we have elliptic curves; elliptic curves are best known for their applications in number theory and cryptography, but they also play an important role in string phenomenology.  The two-dimensional, simply connected, smooth varieties with trivial canonical bundle are named \emph{K3 surfaces}, after the mathematicians Kummer, K\"{a}hler, and Kodaira and the mountain K2.  Like elliptic curves, K3 surfaces are all diffeomorphic to each other, but possess rich complex and arithmetic structure.

The Type IIA/Type IIB physical duality may be realized mathematically using many different constructions, using objects as disparate as hypersurfaces in toric varieties and the bounded derived category of coherent sheaves.  We will compare two different methods of describing mirrors of K3 surfaces.  These constructions provide a strategy for identifying interesting examples of K3 surfaces with high Picard rank. 

\section{Reflexive polytopes and toric hypersurfaces}
\label{sec:hypersurfaces}

\subsection{Constructing Calabi-Yau hypersurfaces}

Mirror symmetry takes its name from a particular symmetry observed in the Hodge diamonds of pairs of Calabi-Yau threefolds.  If $X$ is a Calabi-Yau threefold, then the Hodge diamond of $X$ has the following form: 

\[
\begin{matrix}
&\;  & \; &1& \; &\; &\; \\
&\;  &0 & \; & 0&\; &\; \\
& 0 &&h^{1,1}(V)&&0&\\
1&&h^{2,1}(V)&&h^{2,1}(V)&&1 \\
& 0 &&h^{1,1}(V)&&0&\\
& &0 && 0&&\\
& & &1& &&
\end{matrix}
\]

If $X^\circ$ is a Calabi-Yau threefold mirror to $X$, then $h^{1,1}(X) = h^{2,1}(V^\circ)$ and $h^{2,1}(X) = h^{1,1}(X^\circ)$.  Thus, the Hodge diamonds of $X$ and $X^\circ$ are related by a reflection across a diagonal line.

In \cite{Batyrev}, Batyrev showed that one can use a combinatorial duality between certain types of lattice polytopes to generate mirror families of Calabi-Yau threefolds.  Let us recall the construction here.  (A more detailed exposition may be found in \cite{CoxKatz}.)

We will realize our Calabi-Yau threefolds as hypersurfaces in certain toric varieties.  A \emph{lattice} is a finitely generated free abelian group equipped with a $\mathbb{Z}$-bilinear pairing.  Let $N \cong \mathbb{Z}^k$ be a lattice.  Associated with $N$ is a dual lattice $M \cong \mathrm{Hom}(N, \mathbb{Z})$.  (This reverse alphabetic convention is generally explained by the mnemonic that $N$ is where we find a faN and $M$ is where we find Monomials; readers new to toric geometry should bear this slogan in mind as we explore the details of our construction.)  The natural pairing $\langle v, w\rangle$ of elements of $N$ and $M$ respectively extends to a real-valued pairing of elements of the associated vector spaces, $N_\mathbb{R}$ and $M_\mathbb{R}$. 

We define a \emph{polytope} in a finite-dimensional real vector space as the convex hull of a finite set of points.  A \emph{lattice polytope} is a polytope in the vector space $N_\mathbb{R} = N \otimes \mathbb{R}$ with vertices in $N$. 

\begin{definition}Let $\Delta$ be a lattice polytope in $N_\mathbb{R}$ which contains $(0,\dots,0)$.  The \emph{polar polytope} $\Delta^\circ$ is the polytope in $M_\mathbb{R}$ given by: 

\begin{align*}
\{(m_1, \dots, m_k) &: \langle (n_1, \dots, n_k) , (m_1, \dots, m_k) \rangle \geq -1 \; \\
& \mathrm{for}\;\mathrm{all}\;(n_1,\dots,n_k) \in \Delta\}
\end{align*}
\end{definition}

We say a lattice polytope $\Delta$ is \emph{reflexive} if its polar polytope $\Delta^\circ$ is also a lattice polytope.  If $\Delta$ is reflexive, $(\Delta^\circ)^\circ = \Delta$, and we say $\Delta$ and $\Delta^\circ$ are a \emph{mirror pair}.  We illustrate a two-dimensional pair of reflexive polytopes in Figures~\ref{F:2dsimplex} and \ref{F:polar}. 

\begin{figure}[ht]
\begin{center}
\scalebox{.8}{\includegraphics{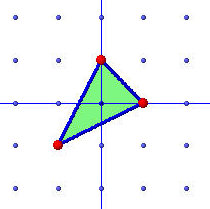}}
\end{center}
\caption{A reflexive triangle}\label{F:2dsimplex}
\end{figure}

\begin{figure}[ht]
\begin{center}
\scalebox{.8}{\includegraphics{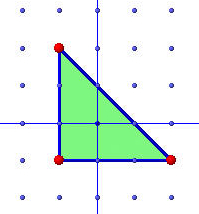}}
\end{center}
\caption{Our triangle's polar polygon}
\label{F:polar}
\end{figure}

The physicists Kreuzer and Skarke classified three- and four-dimensional reflexive polytopes up to overall lattice isomorphisms. \cite{KS}  The number of reflexive polytopes up to overall isomorphism in each dimension is shown in Table~\ref{Ta:reflexive}.

\begin{table}
\begin{center}
\begin{tabular}{|c|c|}
\hline
\textbf{Dimension} & \textbf{Reflexive Polytopes}\\ \hline
1 & {1} \\ \hline
2 & {16} \\ \hline
3 & {4,319} \\ \hline
4 & {473,800,776} \\ \hline
5 & {??} \\ \hline
\end{tabular}
\caption{Classification of Reflexive Polytopes}\label{Ta:reflexive}
\end{center}
\end{table}

Toric varieties are constructed using the combinatorial data of a \emph{fan}, which is constructed by gluing together \emph{cones}.  A cone in $N$ is a subset of the real vector space $N_\mathbb{R}$ generated by nonnegative $\mathbb{R}$-linear combinations of a set of vectors $\{v_1, \dots , v_m\} \subset N$.  We assume that cones are strongly convex, that is, they contain no line through the origin.  A fan $\Sigma$ consists of a finite collection of cones such that each face of a cone in the fan is also in the fan, and any pair of cones in the fan intersects in a common face. We say a fan $\Sigma$ is \emph{simplicial} if the generators of each cone in $\Sigma$ are linearly independent over $\mathbb{R}$; we say $\Sigma$ is \emph{smooth} if the generators of each cone in $\Sigma$ are a $\mathbb{Z}$-basis for $N$. 

Given a lattice polytope containing the origin, we may use the polytope to construct a fan $\Sigma$ in several ways.  First, we may take the fan $R$ over the faces of a lattice polytope $\Delta$ with vertices in $N$.  In this construction, each $j$-dimensional face of $\Delta$ yields a $j+1$-dimensional cone in $\Sigma$; in particular, the vertices of $\Delta$ correspond to the one-dimensional cones of $\Sigma$, and the facets of $\Delta$ correspond to the $k$-dimensional cones of $\Sigma$.  In Figure~\ref{F:facefan}, we illustrate the fan over the faces of the reflexive simplex from Figure~\ref{F:2dsimplex}.

\begin{figure}[ht]
\begin{center}
\scalebox{.8}{\includegraphics{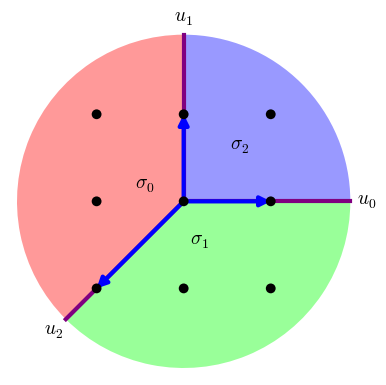}}
\end{center}
\caption{Fan over the faces of the simplex}\label{F:facefan}
\end{figure}

Although the fan $R$ is easy to describe, it may not have all of the properties we desire in a fan; for instance, it may not be simplicial or smooth.  Thus, we may wish to \emph{refine} $R$ by adding one-dimensional cones corresponding to other lattice points on the boundary of $\Delta$, and subdividing the cones of $R$ appropriately.  In particular, a simplicial refinement $\Sigma$ of fan $R$ such that the one-dimensional cones of $\Sigma$ are precisely the nonzero lattice points $v_1, \dots, v_q$ of $\Delta$ is called a \emph{maximal projective subdivision} of $R$.

Alternatively, instead of starting with a polytope in $N$, we may take the \emph{normal fan} $S$ to a polytope $E$ with vertices in $M$.  The normal fan associates a $k-j$-dimensional cone $\sigma_f$ to each $j$-dimensional face $f$ of $E$, using the rule
\[ \sigma_f = \{ u \in N_\mathbb{R} | \langle u, v \rangle \leq \langle u, v' \rangle \text{ for all } v \in f \text{ and } v' \in E\}. \] 
In Figure~\ref{F:normalfan}, we illustrate the normal fan to the reflexive simplex from Figure~\ref{F:2dsimplex}.  We illustrate a refinement of this fan in Figure~\ref{F:refinement}. 

\begin{figure}[ht]
\begin{center}
\scalebox{.8}{\includegraphics{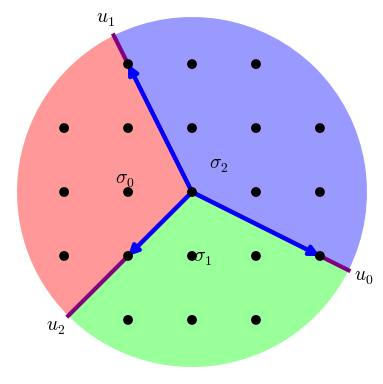}}
\end{center}
\caption{Normal fan to the simplex}\label{F:normalfan}
\end{figure}

\begin{figure}[ht]
\begin{center}
\scalebox{.8}{\includegraphics{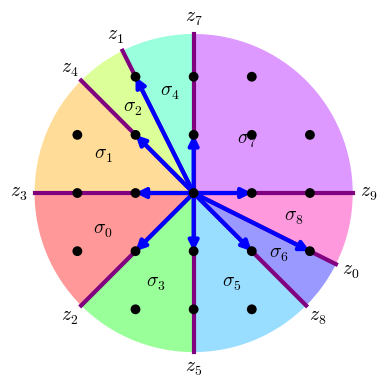}}
\end{center}
\caption{Refined fan}\label{F:refinement}
\end{figure}

If $E$ is a reflexive polytope, then the normal fan to $E$ is the fan over the faces of $E^\circ \in N_\mathbb{R}$; thus, when working with reflexive polytopes, the notions are equivalent.

We may use a fan $\Sigma$ to construct a toric variety $V_\Sigma$.  We will describe $V_\Sigma$ using homogeneous coordinates.  This method generalizes the construction of $\mathbb{P}^n$ as a quotient space of $(\mathbb{C}^*)^n$.  For a more detailed exposition, the reader should consult \cite[Chapter 5]{CLS}.  Let $\Sigma(1) = \{\rho_1, \dots, \rho_q\}$ be the set of one-dimensional cones of $\Sigma$.  For each $\rho_j$ in $\Sigma(1)$, let $v_j$ be the unique generator of the additive semigroup $\rho_j \cap N$.  To each edge $\rho_j \in \Sigma(1)$, we associate a coordinate $z_j$, for a total of $q$ coordinates.   

The toric variety $V_\Sigma$ will be a $k$-dimensional quotient of a subset of $\mathbb{C}^q$.  Let $\mathcal{S}$ denote any subset of $\Sigma(1)$ that does \emph{not} span a cone of $\Sigma$.  Let $\mathcal{V}(\mathcal{S})\subseteq \mathbb{C}^q$ be the linear subspace defined by setting $z_j = 0$ for each $\rho_j\in \mathcal{S}$.  Let $Z(\Sigma)$ be the union of the spaces $\mathcal{V}(\mathcal{S})$.  Observe that $(\mathbb{C}^*)^q$ acts on $\mathbb{C}^q - Z(\Sigma)$ by coordinatewise multiplication.  Fix a basis for $N$, and suppose that $v_j$ has coordinates $(v_{j 1}, \dots, v_{j n})$ with respect to this basis.  Consider the homomorphism $\phi : (\mathbb{C}^*)^q \to (\mathbb{C}^*)^k$ given by
\[ \phi(t_1, \dots, t_q) \mapsto \left( \prod_{j=1}^q t_j ^{v_{j1}} , \dots, \prod_{j=1}^q t_j^{v_{j k}} \right) \]
The toric variety $V_\Sigma$ associated with the fan $\Sigma$ is given by the quotient
\[ V_\Sigma  = (\mathbb{C}^q - Z(\Sigma)) / \text{Ker}(\phi).\]

Now, let us fix a $k$-dimensional reflexive polytope $\Delta$ in a lattice $N$, and let $R$ be the fan constructed by taking cones over the faces of $\Delta$.  A generic representative $Y$ of the anticanonical class of the toric variety $\mathcal{V}(R)$ will be a Calabi-Yau variety.  However, in general neither $Y$ nor the ambient toric variety $\mathcal{V}(R)$ will be smooth, or even an orbifold.  Thus, we will work with a maximal simplicial refinement $\Sigma$ of $R$.   The refinement induces a birational morphism of toric varieties $f: \mathcal{V}(R) \to \mathcal{V}(\Sigma)$.  The toric variety $\mathcal{V}(\Sigma)$ is a Gorenstein orbifold with at worst terminal singularities.  The map $f$ is \emph{crepant}, that is, the pullback of the canonical class of $\mathcal{V}(\Sigma)$ under $f$ is the canonical class of $\mathcal{V}(R)$.  Thus, $f$ yields a relationship between the Calabi-Yau varieties $Y$ in $\mathcal{V}(R)$ and the generic representatives $X$ of the anticanonical class of $\mathcal{V}(\Sigma)$.  The varieties $X$ are \emph{minimal Calabi-Yau orbifolds}, that is, they are Calabi-Yau varieties which are Gorenstein orbifolds with at worst terminal singularities. (A proof of this fact may be found in \cite{CoxKatz}.)

Using the homogeneous coordinates construction, we may write the Calabi-Yau varieties $X$ explicitly as the vanishing sets of certain polynomials $p$:
\begin{equation}
p = \sum_{x \in \Delta^\circ \cap M} c_x \prod_{j=1}^q z_j^{\langle v_j, x \rangle + 1}.
\end{equation}
Here the $v_j$ are lattice generators of the one-dimensional cones of $\Sigma$, and the $z_j$ are the corresponding homogeneous coordinates. 

\begin{example}
The simplex with vertices $(1,0,\dots,0)$, $(0,1,0,\dots,0)$, \dots, $(0,\dots,0,1)$ and $(-1,\dots,-1)$ corresponds to the toric variety $\mathbb{P}^k$.  Its polar dual has vertices at $(k,-1,\dots,-1)$, $(-1,k,-1,\dots,-1)$, \dots, $(-1,\dots,-1,3)$, and $(-1,-1,\dots,-1)$.  The resulting polynomials are homogeneous polynomials of degree $k+1$ in $k+1$ variables.
\end{example}

The singular locus of a variety with at worst Gorenstein terminal singularities has codimension at least four.  Thus, for $k=3$ both the ambient toric variety $\mathcal{V}(\Sigma)$ and the hypersurfaces $X$ will be smooth; this implies that the fan $\Sigma$ must be not just simplicial, but smooth.  In this case, the hypersurfaces are K3 surfaces.  When $k=4$, $\mathcal{V}(\Sigma)$ may have orbifold singularities, but the Calabi-Yau threefolds $X$ will still be smooth. 

In a more general setting, we may measure failures of smoothness using the polynomial $p$.  If the partial derivatives $\partial p/\partial z_j$, $j = 1, \dots, q$ do not vanish simultaneously on $X$, we say $X$ is \emph{quasismooth}.  When $X$ is quasismooth, its only singularities are inherited from singularities of the ambient toric variety.  If $X$ satisfies the stronger condition that the partial derivatives $z_j \, \partial p/\partial z_j$, $j = 1, \dots, q$ do not vanish simultaneously on $X$, we say $X$ is \emph{regular} and $p$ is \emph{nondegenerate}.  In this case, $X$ intersects the coordinate hypersurfaces $z_j=0$ transversely. 

\subsection{Counting Hodge numbers on Calabi-Yau hypersurfaces}

We may use the combinatorial data of the polytope $\Delta$ to study the Hodge numbers of $X$.  The nonzero lattice points $v_j$ of $\Delta$ correspond to irreducible torus-invariant divisors $W_k$ in $\mathcal{V}(\Sigma)$.  (In global homogeneous coordinates, these are just the hypersurfaces $z_j=0$.)  Because $\Sigma$ is simplicial, the divisors $W_j$ generate $\mathrm{Pic}(\mathcal{V}(\Sigma)) \otimes \mathbf{Q}$ subject to certain relations; in particular, $\mathrm{rank}\;\mathrm{Pic}(\mathcal{V}(\Sigma)) = q-3$.  When $V$ is smooth, the divisors generate $\mathrm{Pic}(\mathcal{V}(\Sigma))$.

Let $X$ be a regular Calabi-Yau hypersurface in $\mathcal{V}(\Sigma)$.  Generically, the intersection of a divisor $W_k$ with $X$ is empty when the corresponding lattice point $v_j$ is in the interior of a codimension-one face of $\Delta$.  If $v_j$ is on the boundary of a codimension-one face, then the intersection of $W_j$ and $X$ may form a single divisor of $X$; alternatively, $W_j \cap X$ may split into several irreducible components.  Specifically, $W_j$ splits when the corresponding lattice point $v_j$ is interior to a codimension-two face $\Gamma$ of $\Delta$ and the dual face $\Gamma^\circ$ also has interior points.  In this case, $W_j \cap X$ has $\ell(\Gamma^\circ)-1$ components $W_{ji}$, where $\ell(\Gamma^\circ)$ is the number of lattice points in the dual face $\Gamma^\circ$.  

Batyrev used these counts to show that, when $k \geq 4$, the Hodge number $h^{1,1}(X)$ is given by the following formula: \cite{Batyrev}

\begin{equation} h^{1,1}(X) = \ell(\Delta) - k - 1 - \sum_{\mathrm{codim}\; \Gamma = 1} \ell^*(\Gamma) + \sum_{\mathrm{codim}\; \Gamma = 2} \ell^*(\Gamma) \ell^*(\Gamma^\circ) \label{E:h11} \end{equation}

\noindent Here $\Gamma$ is a face of $\Delta$ of the given codimension, $\Gamma^\circ$ is the dual face of $\Delta^\circ$, and $\ell^*(\Gamma)$ is the number of points in the relative interior of the face.

The Hodge number $h^{k-2,1}(X)$ measures the dimension of the vector space describing infinitesimal variations of complex structure of $X$.  Because each lattice point in $\Delta^\circ$ determines a monomial in the polynomial which defines $X$, one may compute $h^{k-2,1}(X)$ in terms of the lattice points of $\Delta^\circ$: \cite{Batyrev}

\begin{equation}
h^{k-2,1}(X) = \ell(\Delta^\circ) - k - 1 - \sum_{\mathrm{codim}\; \Gamma^\circ = 1} \ell^*(\Gamma^\circ) + \sum_{\mathrm{codim} \; \Gamma^\circ = 2} \ell^*(\Gamma^\circ) \ell^*(\Gamma)
\end{equation}

Let $X^\circ$ be a generic Calabi-Yau hypersurface in the family obtained from $\Delta^\circ$.  Interchanging the roles of $\Delta$ and $\Delta^\circ$ in the above formulas, we obtain the following theorem:

\begin{theorem}\label{T:batyrevMirrorSymmetry}\cite{Batyrev}
For $k \geq 4$, the Hodge numbers of $X$ and $X^\circ$ are related by $h^{1,1}(X) = h^{k-2,1}(X^\circ)$ and $h^{k-2,1}(X) = h^{1,1}(X^\circ)$.
\end{theorem}

\section{Mirror symmetry for K3 surfaces}

\subsection{Lattice-polarized K3 surfaces}

The two-dimensional smooth Calabi-Yau varieties are known as \emph{K3 surfaces}.  All K3 surfaces are diffeomorphic, with Hodge diamond given by:

\begin{equation}\begin{matrix}
& & 1 & & \\
& 0 & & 0 & \\
1 & & 20 & &  1 \\
& 0 & & 0 & \\
& & 1 & & 
\end{matrix}
\end{equation}

We see immediately that a meaningful notion of mirror symmetry for K3 surfaces cannot depend on mere correspondences of Hodge numbers!  Instead of studying the full vector space $H^{1,1}(X, \mathbb{C})$ for a K3 surface $X$, we will focus on an important subset of this space.

Any K3 surface X has $H^2(X,\mathbb{Z}) \cong L$, where 
$L = U \oplus U \oplus U \oplus E_8 \oplus E_8$
is a lattice of signature $(3,19)$.  Here, $U$ is the rank-two indefinite lattice with pairing given by $\begin{pmatrix}
0 & 1 \\
1 & 0 \end{pmatrix}$.  We call a choice of isomorphism $\phi: H^2(X,\mathbb{Z}) \to L$ a \emph{marking} of $X$, and refer to the pair $(X;\phi)$ as a \emph{marked K3 surface}. 

The \emph{Picard group} $\mathrm{Pic}(X)$ is given by the intersection $H^{1,1}(X, \mathbb{C}) \cap H^2(X,\mathbb{Z})$.  The Picard group is a free abelian group; we refer to the rank of this group as the \emph{Picard rank}, and write it as $\text{rank } \mathrm{Pic}(X)$. The Picard group inherits a lattice structure from the lattice structure on $L$.  We may identify $\mathrm{Pic}(X)$ with the N\'{e}ron-Severi group of algebraic curves using Poincar\'{e} duality; from this point of view, the lattice structure corresponds to the intersection product. 

If we move within a family of K3 surfaces, the corresponding Picard groups may change in a discontinuous fashion.  Indeed, Oguiso showed that any analytic neighborhood in the base of a one-parameter,
non-isotrivial family of K3 surfaces has a dense subset where the Picard ranks of
the corresponding surfaces are greater than the minimum Picard rank of the family. \cite{OguisoPic}

In \cite{Dolgachev}, Dolgachev argued that we should study mirror symmetry for K3 surfaces using the notion of a \emph{lattice-polarized} K3 surface.  Let $M$ be an even, nondegenerate lattice of signature $(1,t)$.  We assume that $t \leq 19$. 

\begin{definition}\cite{Dolgachev}
An $M$-\emph{polarized} K3 surface $(X,i)$ is a K3 surface $X$ together with a lattice embedding $i: M \hookrightarrow \mathrm{Pic}(X)$.  We require that the embedding $i$ be primitive, that is, $\mathrm{Pic}(X)/i(M)$ must be a free abelian group.
\end{definition}

Let $U(m)$ be the lattice with intersection matrix $\left(\begin{smallmatrix} 0 & m\\ m & 0\end{smallmatrix}\right)$, where $m$ is a positive integer.

\begin{definition}\cite{Dolgachev}
We say a sublattice $M$ of $L$ is $m$\emph{-admissible} if $M^\perp = J \oplus \check{M}$, where $J$ is isomorphic to $U(m)$.  In this situation, we call $\check{M}$ the \emph{mirror} of $M$.
\end{definition}

\noindent Note that if $M$ and $\check{M}$ are mirror lattices, then $\mathrm{rank} \; M + \mathrm{rank} \; \check{M} = 20$.

In Dolgachev's framework, mirror families of K3 surfaces contain K3 surfaces which are polarized by mirror lattices.  Dolgachev constructs a moduli space $\mathbf{K}_M$ of marked, $M$-polarized K3 surfaces.  This space has dimension $20-\mathrm{rank} \; M$, so we may also write $\mathrm{dim}\;\mathbf{K}_M+\mathrm{dim}\;\mathbf{K}_{\check{M}}=20$.  \cite{Dolgachev}

\subsection{K3 surfaces as toric hypersurfaces}

Are K3 surfaces realized as hypersurfaces in toric varieties obtained from polar dual reflexive polytopes mirror to each other in the sense of Dolgachev?  In order to answer this question, we must identify lattice polarizations of the K3 surfaces.  Let us fix $k=3$, let $\Delta$ be a reflexive polytope in $N$, and let $\Sigma$ be a maximal simplicial refinement of the fan $R$ over the faces of $\Delta$.  Let $\iota: X \to \mathcal{V}(\Sigma)$ be the inclusion map; we define the so-called \emph{toric divisors} as $\mathrm{Pic}_\mathrm{tor}(X) = \iota^*(\mathrm{Pic}(V))$.  Let us set $\rho = \mathrm{rank}\;\mathrm{Pic}_\mathrm{tor} (X)$.  We shall refer to the sum $\delta = \sum_{\mathrm{codim} \Gamma = 2} \ell^*(\Gamma) \ell^*(\Gamma^\circ)$ as the \emph{toric correction term}.  The toric divisors together with the divisors $W_{kj}$ that arise from splitting generate a group of rank $\rho+\delta$ which we shall call $\mathrm{Pic}_\mathrm{cor}(X)$.   

We see that in the case of K3 surfaces, the equality of Equation~\ref{E:h11} is replaced by an inequality:

\begin{lemma}\label{P:minimumPicRank}
Let $X$ be a regular K3 hypersurface in $\mathcal{V}(\Sigma)$.  Then,
\[\mathrm{rank}\;\mathrm{Pic}_\mathrm{tor} (X) = \ell(\Delta) - 4 - \sum_{\mathrm{codim} \Gamma = 1} \ell^*(\Gamma)\]
and
\[ \mathrm{rank}\;\mathrm{Pic}(X) \geq \mathrm{rank}\;\mathrm{Pic}_\mathrm{tor} (X) + \delta.\]
\end{lemma}

One might hypothesize that if $X$ and $X^\circ$ are generic K3 surfaces in toric varieties corresponding to polar dual reflexive polytopes $\Delta$ and $\Delta^\circ$, then $\mathrm{Pic}_\mathrm{cor}(X)$ and $\mathrm{Pic}_\mathrm{cor}(X^\circ)$ are mirror lattices.  However, one can show by direct computation that this is impossible, as $(\rho(X) + \delta) + (\rho(X^\circ) + \delta)$ need not be $20$.  We illustrate this fact in Table~\ref{Ta:rhodelta}, using the data for the first 10 three-dimensional reflexive polytopes in Kreuzer and Skarke's database. \cite{KS, Rohsiepe}

\begin{table}[ht]
\begin{center}
\begin{tabular}{|c | c| c| c| c|} 
\hline
$\Delta$ & $\Delta^\circ$ &  $\rho(\Delta)$ & $\rho(\Delta^\circ)$ & $\delta(\Delta)$\\
\hline
0 & 4311 & 1 & 19 & 0\\
1 & 4281 & 4 & 18 & 2\\
2 & 4317 & 1 & 19 & 0\\
3 & 4283 & 2 & 18 & 0\\
4 & 4286 & 2 & 18 & 0\\
5 & 4296 & 2 & 18 & 0\\
6 & 4285 & 2 & 18 & 0\\
7 & 4309 & 2 & 18 & 0\\
8 & 3313 & 9 & 17 & 6\\
9 & 4312 & 3 & 18 & 1\\
\hline
\end{tabular}
\caption{Comparison of $\mathrm{Pic}_\mathrm{cor}(X)$ and $\mathrm{Pic}_\mathrm{cor}(X^\circ)$}\label{Ta:rhodelta}
\end{center}
\end{table}

Rohsiepe observed that instead of using the full lattice $\mathrm{Pic}_\mathrm{cor}$ on both sides, we can view $X$ as polarized by $\mathrm{Pic}_\mathrm{cor}(X)$ and $X^\circ$ as polarized by $\mathrm{Pic}_\mathrm{tor}(X^\circ)$.  Alternatively, we can reverse our view of which is the starting polytope, and 
view $X$ as polarized by $\mathrm{Pic}_\mathrm{tor}(X)$ and $X^\circ$ as polarized by $\mathrm{Pic}_\mathrm{cor}(X^\circ)$.

\begin{theorem}\cite{Rohsiepe}\label{T:mirrorK3}
Let $X$ and $X^\circ$ be regular K3 surfaces in toric varieties obtained from polar dual three-dimensional reflexive polytopes $\Delta$ and $\Delta^\circ$, respectively.  Then $\mathrm{Pic}_\mathrm{cor} (X)$ and $\mathrm{Pic}_\mathrm{tor}(X^\circ)$ are mirror latttices, as are $\mathrm{Pic}_\mathrm{tor} (X)$ and $\mathrm{Pic}_\mathrm{cor}(X^\circ)$.  Furthermore, $(\mathrm{Pic}_\mathrm{tor} (X))^\perp \cong \mathrm{Pic}_\mathrm{cor} (X^\circ) \oplus U$ and $(\mathrm{Pic}_\mathrm{cor} (X))^\perp \cong \mathrm{Pic}_\mathrm{tor} (X^\circ)\oplus U$.  
\end{theorem}

Rohsiepe proved Theorem~\ref{T:mirrorK3} by direct computation of intersection forms on the lattices, using the database of three-dimensional reflexive polytopes computed in \cite{KS}.  The results of his computations may be found in \cite{RohsiepeTables}.

\section{Highly symmetric K3 surfaces}

We may identify interesting families of K3 surfaces with high Picard rank by studying non-generic K3 hypersurfaces in toric varieties obtained from reflexive polytopes.  Our plan is to search for families where the rank of the Picard group is strictly greater than $\rho + \delta$.  Many such families admit interesting group actions.

\subsection{Symplectic group actions}

Let $X$ be a K3 surface, and let $G$ be a finite group acting on $X$ by automorphisms.  
The action of $G$ on $X$ induces an action on the cohomology of $X$.  We say $G$ acts \emph{symplectically} if $G$ acts as the identity on $H^{2,0}(X)$.  Mukai showed that any finite group $G$ with a symplectic action on a K3 surface is a subgroup of a member of a list of eleven groups, and gave an example of a symplectic action of each of these maximal groups. \cite{Mukai}  Xiao gave an alternate proof of the classification by listing the possible types of singularities. \cite{Xiao}  We define a sublattice $S_G$ of $H^2(X,\mathbb{Z})$ as the perpendicular complement of the part of $H^2(X,\mathbb{Z})$ fixed by the induced action of $G$: $S_G = (H^2(X,\mathbb{Z})^G)^\perp$. 

\begin{proposition}\label{P:Nik} \cite{Nikulin}
The lattice $S_G$ is a negative definite sublattice of $\mathrm{Pic}(X)$.
\end{proposition}

In general, $S_G$ may depend on the particular group action of $G$.  However, it follows from the results of \cite[\S 3]{Mukai}, that the rank of $S_G$ depends only on the group $G$.  One may compute this rank from the data given in \cite{Xiao}; for a detailed procedure, see \cite{W}.  Moreover, Hashimoto proved in \cite{HashimotoSymp} that $S_G$ is unique up to overall lattice isomorphism for all but five of the 81 groups which can act symplectically.

Demazure and Cox showed that the automorphism group $A$ of a $k$-dimensional toric variety $\mathcal{V}(\Sigma)$ is generated by the big torus $T \cong (\mathbb{C}^*)^k$, symmetries of the fan $\Sigma$ induced by lattice automorphisms, and one-parameter families derived from the ``roots'' of $\mathcal{V}(\Sigma)$. \cite{CoxKatz}  We are interested in finite subgroups of $A$ which act symplectically on K3 hypersurfaces $X$ in 3-dimensional toric varieties $\mathcal{V}(\Sigma)$ obtained from reflexive polytopes.

To determine when a subgroup acts symplectically, we need an explicit description of a generator of $H^{2,0}(X)$.  We realize this form as the residue of a form defined on $\mathcal{V}(\Sigma) - X$.

\begin{proposition}\label{P:Mav}\cite{Mavlyutov}
Let $X$ be a regular K3 hypersurface in $\mathcal{V}(\Sigma)$ described in homogeneous coordinates by a polynomial $p$.  Choose an integer basis $m_1,\dots,m_n$ for the dual lattice $M$.  For any $n$-element subset $I = \{i_1, \dots, i_n\}$ of $\{1,\dots,q\}$, let $\mathrm{det}\,v_I = \mathrm{det}\,(\langle m_j, v_{i_k} \rangle_{1 \leq j, i_k \leq n})$, $dz_I = dz_{i_1} \wedge \dots \wedge dz_{i_n}$, and $\hat{z}_I = \prod_{i \notin I}z_i$.  Let $\Omega$ be the 3-form on $\mathcal{V}(\Sigma)$ given in global homogeneous coordinates by $\sum_{|I|=n}\mathrm{det}\,v_I \hat{z}_I dz_I$.  Then $\omega := \mathrm{Res}(\Omega/p)$ generates $H^{2,0}(X)$.
\end{proposition}

\subsection{Big torus actions}

We begin by analyzing finite subgroups of the big torus $T$.  

\begin{proposition}\label{P:SympAction}
Let $X$ be a regular K3 hypersurface in $\mathcal{V}(\Sigma)$ described in homogeneous coordinates by a polynomial $p$, and represent $g \in T$ by a diagonal matrix $D \in GL(q, \mathbf{C})$.  Suppose $g^* p = \lambda \, p$, $\lambda \in \mathbf{C}^*$, and $\mathrm{det}(D) = \lambda$.  Then the induced action of $g$ on the cohomology of $X$ fixes the holomorphic 2-form $\omega$ of $X$.
\end{proposition}

\begin{proof}
Let $\Omega$ be the 3-form on $V$ defined in Proposition~\ref{P:Mav}.  Then $g^* (\Omega) = \mathrm{det}(\Delta) \, \Omega$, so $g^* (\Omega/p) = (\lambda/\lambda) (\Omega/p) = (\Omega/p)$.  Thus, $g$ fixes the generator $\mathrm{Res}(\Omega/p)$ of $H^{2,0}(X)$.
\end{proof}

\begin{remark}
If $\mathcal{V}(\Sigma) = \mathbb{P}^3$, then $g^* \Omega = \mathrm{det}(D) \, \Omega$ for any automorphism $g$ of $\mathcal{V}(\Sigma)$ induced by a matrix $D \in GL(4, \mathbb{C})$; cf. \cite[Lemma 2.1]{Mukai}.
\end{remark}

K3 hypersurfaces which admit finite torus actions have enhanced Picard rank.

\begin{proposition}\label{P:LowerBoundRankPic}
Let $X$ be a regular K3 hypersurface in $\mathcal{V}(\Sigma)$.  Let $G$ be a finite subgroup of $T$ which acts symplectically on $X$.  Then,  
\[\mathrm{rank \; Pic}(X) \geq \mathrm{rank} \; \mathrm{Pic}_\mathrm{cor} (X) + \mathrm{rank} \; S_G.\]
\end{proposition}

\begin{proof}
Since $G$ is a subgroup of $T$, the divisors $W_k$ of $\mathcal{V}(\Sigma)$ are stable under the action of $G$.  Therefore, $\mathrm{Pic}_\mathrm{cor}(X) \subseteq H^2(X,\mathbb{Z})^G$.  The proposition then follows from the facts that $S_G = (H^2(X,\mathbb{Z})^G)^\perp$ and that $S_G$ is negative definite.
\end{proof}

As an example, let us consider smooth quartics in $\mathbb{P}^3$.  
The projective space $\mathbb{P}^3$ corresponds to the reflexive polytope $\Delta$ with vertices $(1,0,0)$, $(0,1,0)$, $(0,0,1)$ and $(-1,-1,-1)$.  The only other lattice point of $\Delta$ is $(0,0,0)$, so the toric correction term $\delta$ vanishes, and we have $\mathrm{rank} \; \mathrm{Pic}_\mathrm{cor} (X) = \mathrm{rank} \; \mathrm{Pic}_\mathrm{tor} (X)= 4-3 = 1$.  The mirror K3 surfaces have $\mathrm{rank} \; \mathrm{Pic}_\mathrm{cor} (X^\circ) = \mathrm{rank} \; \mathrm{Pic}_\mathrm{tor} (X^\circ)= 19$.  Now, let us consider the pencil of quartics in $\mathbb{P}^3$ described by $x^4 + y^4 + z^4 + w^4 -4t(xyzw) = 0$.  For generic $t$, the corresponding hypersurface $X_t$ is a regular K3 surface. The group $(\mathbb{Z}/(4\mathbb{Z}))^2$ acts on $X$ by $x \mapsto \lambda x$, $y \mapsto \mu y$, $z \mapsto \lambda^{-1} \mu^{-1} z$, where $\lambda$ and $\mu$ are fourth roots of unity.  By Proposition~\ref{P:SympAction}, this action is symplectic.   Nikulin showed that $\mathrm{rank}\;S_G = 18$. \cite{Nikulin}  Thus, by Proposition~\ref{P:LowerBoundRankPic}, $\mathrm{rank} \; \mathrm{Pic}(X_t) \geq 1+18 = 19$.

Similarly, $\mathbb{WP}(1,1,1,3)$ corresponds to the reflexive polytope $\Delta$ with vertices $(1,0,0)$, $(0,1,0)$, $(0,0,1)$ and $(-1,-1,-3)$.  The only other lattice points of $\Delta$ are $(0,0,-1)$, which is interior to a face, and the origin.  Thus, $\mathrm{rank} \; \mathrm{Pic}_\mathrm{cor} (X) = \mathrm{rank} \; \mathrm{Pic}_\mathrm{tor} (X) = 4-3 = 1$.  However, if we restrict our attention to the diagonal pencil of K3 surfaces $X_t$ in $\mathbb{WP}(1,1,1,3)$ given by $x^6 + y^6 + z^6 + w^2 -txyzw = 0$, we discover that if $X_t$ is a regular K3 surface in this family, then $\mathrm{rank} \; \mathrm{Pic}(X) \geq 19$.  In this case, the group $\mathbb{Z}/(6\mathbb{Z}) \times \mathbb{Z}/(2\mathbb{Z})$ acts on $X_t$ by $x \mapsto x$, $y \mapsto \lambda y$, $z \mapsto \lambda^{-1} \mu^{-1} z$, and $w \mapsto \mu w$, where $\lambda$ is a sixth and $\mu$ a square root of unity.  By Proposition~\ref{P:SympAction}, this action is symplectic.   Nikulin showed that $\mathrm{rank}\;S_G = 18$. \cite{Nikulin}  Thus, by Proposition~\ref{P:LowerBoundRankPic}, $\mathrm{rank} \; \mathrm{Pic}(X_t) \geq 1+18 = 19$.

\subsection{Fan Symmetries}\label{SS:fansymmetries}

Let us now consider the automorphisms of $\mathcal{V}(\Sigma)$ induced by symmetries of the fan $\Sigma$, following the discussion in \cite{KLMSW}.  Since $\Sigma$ is a refinement of $R$, the fan consisting of cones over the faces of $\Delta$, the group of symmetries of $\Sigma$ must be a subgroup $H'$ of the group $H$ of symmetries of $\Delta$ (viewed as a lattice polytope).  Let us identify a family $\mathcal{F}_\Delta$ of K3 surfaces in $V$ on which $H'$ acts by automorphisms.

Let $h \in H'$, and let $X$ be a K3 surface in $V$ defined by a polynomial $p$ in global homogeneous coordinates.  Then $h$ maps lattice points of $\Delta$ to lattice points of $\Delta$, so we may view $h$ as a permutation of the global homogeneous coordinates $z_i$: $h$ is an automorphism of $X$ if $p \circ h = p$.  Alternatively, since $H$ is the automorphism group of both $\Delta$ and its polar dual polytope $\Delta^\circ$, we may view $h$ as an automorphism of $\Delta^\circ$: from this vantage point, we see that $h$ acts by a permutation of the coefficients $c_x$ of $p$, where each coefficient $c_x$ corresponds to a point $x \in \Delta^\circ$.  Thus, if $h$ is to preserve $X$, we must have $c_x = c_y$ whenever $h(x) = y$.  We may define a family of K3 surfaces fixed by $H'$ by requiring that $c_x = c_y$ for any two lattice points $x,y \in \Delta^\circ$ which lie in the same orbit of $H'$:

\begin{proposition}\label{P:FanSymmetryFamily}
Let $\mathcal{F}_\Delta$ be the family of K3 surfaces in $V$ defined by the following family of polynomials in global homogeneous coordinates:
\[ p = (\sum_{q \in \mathscr{O}} c_{q} \sum_{x \in \mathscr{O}} \prod_{k=1}^n z_k^{\langle v_k, x \rangle + 1}) + \prod_{k=1}^n z_k,\]
where $\mathscr{O}$ is the set of orbits of nonzero lattice points in $\Delta^\circ$ under the action of $H'$.  Then $H'$ acts by automorphisms on each K3 surface $X$ in $\mathcal{F}_\Delta$.
\end{proposition}

Let $X$ be a regular K3 surface in the family $\mathcal{F}_\Delta$, and let $h \in H'  \subset \mathbf{GL}(3,\mathbb{Z})$.  Using Proposition~\ref{P:Mav}, we compute that $h^*(\omega) = (\mathrm{det}~h)\omega$.  Thus, $h$ acts symplectically if and only if $h$ is orientation-preserving.  We see that the group $G$ of all orientation-preserving automorphisms of $\Delta$ which preserve $\Sigma$ acts symplectically on regular members of $\mathcal{F}_\Delta$.

One may search for families $\mathcal{F}_\Delta$ which are naturally one-parameter.  We expect that these families will be polarized by lattices of rank $19$, because the dimension of the moduli space of lattice-polarized K3 surfaces and the rank of the lattice add to 20.  
One-parameter families $\mathcal{F}_\Delta$ correspond to three-dimensional reflexive polytopes $\Delta^\circ$ in $M$ which have no lattice points other than their vertices and the origin, and whose group of orientation-preserving automorphisms acts transitively on the vertices.  In the classification of three-dimensional reflexive polytopes, there are precisely four such classes of polytopes.  They correspond to the standard simplex with vertices at $(1,0,0)$, $(0,1,0)$, $(0,0,1)$, and $(-1,-1,-1)$ (Figure~\ref{F:3dsimplex}), the octahedron or cross-polytope with vertices at $(\pm 1,0,0)$, $(0,\pm 1, 0)$, and $(0,0,\pm 1)$ (Figure~\ref{F:octahedron}), the unique reflexive polytope with twelve vertices and fourteen faces (Figure~\ref{F:manyfaces}), and the parallelepiped with vertices at $(1, 0, 0)$, $(0, 1, 0)$, $(0, 0, 1)$, $(-1, 1, 1), (1, -1, -1)$, $(0, 0, -1)$, $(0, -1, 0)$, and $(-1, 0, 0)$ (Figure~\ref{F:skewcube}).

\begin{figure}[ht]
\begin{center}
\scalebox{.5}{\includegraphics{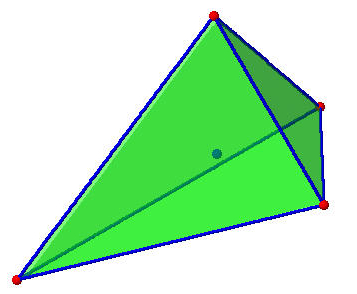}}
\caption{Three-dimensional simplex}\label{F:3dsimplex}
\end{center}
\end{figure}

\begin{figure}[ht]
\begin{center}
\scalebox{.5}{\includegraphics{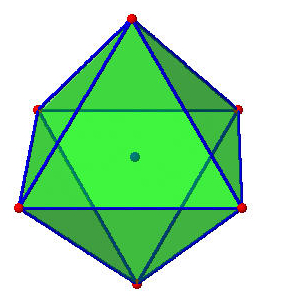}}
\caption{Cross-polytope}\label{F:octahedron}
\end{center}
\end{figure}

\begin{figure}[ht]
\begin{center}
\scalebox{.5}{\includegraphics{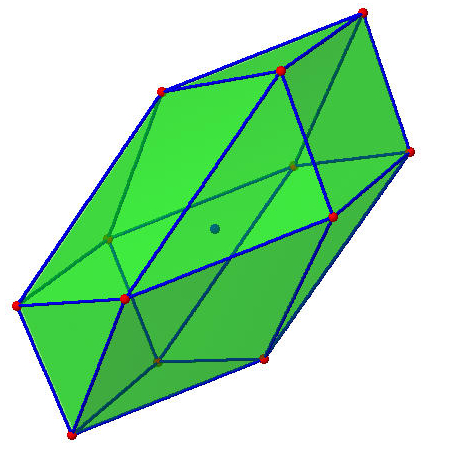}}
\caption{Twelve vertices, fourteen faces}\label{F:manyfaces}
\end{center}
\end{figure}

\begin{figure}[ht]
\begin{center}
\scalebox{.5}{\includegraphics{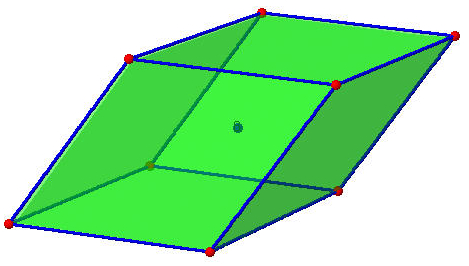}}
\caption{The skew cube}\label{F:skewcube}
\end{center}
\end{figure}

In \cite{KLMSW}, we show that when $\Delta^\circ$ is the standard simplex, the alternating group $\mathcal{A}_4$ acts symplectically, while in the other three cases, the symmetric group $\mathcal{S}_4$ acts symplectically.  By computing $(\mathrm{Pic}_\mathrm{cor}(X))^G$, we then show that each family $\mathcal{F}_\Delta$ is polarized by a lattice of rank 19.  The modularity properties of three of these families had been studied in the literature (see \cite{NS, PS, HLOY, Verrill}).  In \cite{KLMSW}, we compute the Picard-Fuchs equation of the remaining family, corresponding to the parallelepiped in Figure~\ref{F:skewcube}, and show that the holomorphic solution to the Picard-Fuchs equation is a $\Gamma$-modular form, where

\begin{equation} \Gamma = \Gamma_0(4|2) = \left\{ \left. \left( \begin{array}{cc} a & b/2 \\ 4c & d \end{array} \right) \in PSL_2(\mathbb{R}) \: \right| \: a,b,c,d \in \mathbb{Z} \right\}. \end{equation}

\clearpage

\bibliographystyle{plain}
\bibliography{torick3bib}

\end{document}